\documentclass{article}

\usepackage{amssymb}
\usepackage{amsmath}
\usepackage{geometry} 
\usepackage{color}
\usepackage{graphicx}
\usepackage{subcaption}
\usepackage{mathrsfs}
\usepackage[show]{ed}
\usepackage[english,francais]{babel}

\def\R{\mathbb{R}}
\def\Z{\mathbb{Z}}

\def\bu{\mathbf{u}}
\def\bh{\mathbf{h}}
\def\bv{\mathbf{v}}
\def\be{\mathbf{e}}

\newtheorem{theorem}{Theorem}[section]

\newtheorem{e-proposition}[theorem]{Proposition}

\newtheorem{e-definition}[theorem]{Definition\rm}

\newtheorem{proposition}[theorem]{Proposition}

\newtheorem{rem}{\it Remark}

\newcommand{\ri}{{\mathrm i}}
\newcommand{\bn}{{\mathbf n}}
\newcommand{\bx}{{\mathbf x}}

\newcommand{\bU}{{\mathbf U}}
\newcommand{\bH}{{\mathbf H}}

\newcommand\xdiv{\mathrm{div}\,}

\newcommand{\HatOmegaHole}{\hat{\Omega}}

\DeclareMathOperator{\rot}{\bf curl}

\DeclareMathOperator{\dive}{div}

\newcommand{\ie}{\textit{i.e\mbox{.}}}

\def\qtwoperpm{\tilde{q}^\pm_2}
\def\qoneperpm{\tilde{q}^\pm_1}
\makeatletter
\DeclareRobustCommand{\qed}{%
  \ifmmode 
  \else \leavevmode\unskip\penalty9999 \hbox{}\nobreak\hfill
  \fi
  \quad\hbox{\qedsymbol}}
\newcommand{\openbox}{\leavevmode
  \hbox to.77778em{%
  \hfil\vrule
  \vbox to.675em{\hrule width.6em\vfil\hrule}%
  \vrule\hfil}}
\newcommand{\qedsymbol}{\openbox}
\newenvironment{proof}[1][\proofname]{\par
  \normalfont
  \topsep6\p@\@plus6\p@ \trivlist
  \item[\hskip\labelsep\itshape
    #1.]\ignorespaces
}{%
  \qed\endtrivlist
}
\newcommand{\proofname}{Proof}

\def\og{\leavevmode\raise.3ex\hbox{$\scriptscriptstyle\langle\!\langle$~}}
\def\fg{\leavevmode\raise.3ex\hbox{~$\!\scriptscriptstyle\,\rangle\!\rangle$}}

\begin{document}
\centerline{}


\selectlanguage{english}
\title{Electromagnetic shielding by thin periodic structures
and the Faraday cage effect
}


\selectlanguage{english}
\begin{center}
\Large{Electromagnetic shielding by thin periodic structures
and the Faraday cage effect.}
\end{center} 
\begin{center}{\large{B\'erang\`ere Delourme$^{(1)}$, David P.\ Hewett$^{(2)}$}} \end{center} 
\footnotesize{(1)  Laboratoire Analyse G\'eom\'etrie et Applications (LAGA), Universit\'e Paris 13,Villetaneuse, France, delourme@math.univ-paris13.fr }\\
\footnotesize {(2) Department of Mathematics, University College London, London, United Kingdom, d.hewett@ucl.ac.uk}
\normalsize



\medskip

\begin{abstract}

\selectlanguage{english}
\noindent In this note we consider the scattering of electromagnetic waves (governed by the time-harmonic Maxwell equations) by a thin periodic layer of perfectly conducting obstacles. The size of the obstacles and the distance between neighbouring obstacles are of the same small order of magnitude $\delta$. 
By deriving homogenized interface conditions for three model configurations, namely 
(i) discrete obstacles, 
(ii) parallel wires, 
(iii) a wire mesh, 
we show that the limiting behaviour as $\delta\to0$ depends strongly on the topology of the periodic layer, with full shielding (the so-called ``Faraday cage effect'') occurring only in the case of a wire mesh.
\end{abstract}
\vskip 0.5\baselineskip


\selectlanguage{english}
\section{Introduction}\label{}
\noindent 
The ability of wire meshes to block electromagnetic waves (the celebrated ``Faraday cage'' effect) is well known to physicists and engineers. Experimental investigations into the phenomenon date back over 180 years to the pioneering work of Faraday \cite{Faraday}, and the effect is routinely used to block or contain electromagnetic fields in countless practical applications. (An everyday example is the wire mesh in the door of a domestic microwave oven, which stops microwaves escaping, while letting shorter wavelength visible light pass through it.) 
But, somewhat remarkably, a rigorous mathematical analysis of the effect does not appear to be available in the literature. 

The mathematical richness of the Faraday cage effect was highlighted in an recent article by one of the authors \cite{ChapmanHewettTrefethen}, where 
a number of different mathematical approaches were applied to the 
2D electrostatic version of the problem. 
In particular it was shown in \cite{ChapmanHewettTrefethen} how modern techniques of homogenization and matched asymptotic expansions could be used to derive effective interface conditions that accurately capture the shielding effect. 
These results were generalised to the 
2D electromagnetic case (TE- and TM polarizations) in \cite{HewettHewitt}, and related approximations for similar problems have also been studied recently by other authors, e.g.\ \cite{Holloway,MarigoMaurel}.
However, as far as we are aware, an analysis of the full 3D electromagnetic version of the problem with perfectly conducting scatterers (the focus of the current note) has not previously been performed. (We note that related approximations have been presented for thin layers of dielectric obstacles in \cite{DelourmeWellPosedMax,DelourmeHighOrderMax}.) 

In this note we consider full 3D electromagnetic scattering by a thin periodic layer of small, perfectly conducting obstacles. We derive leading-order homogenized interface conditions for three model configurations, namely where the periodic layer comprises (i) discrete obstacles, 
(ii) parallel wires, and (iii) a wire mesh. 
Our results verify that 
the effective behaviour depends strongly on the topology of the periodic layer, with shielding of arbitrarily polarized waves occurring only in the case of a wire mesh. We note that analogous observations have been made in the related setting of volume homogenization in \cite{SchweizerUrban}.



Our analysis assumes that the obstacles/wires making up the thin periodic layer are of approximately the same size/thickness as the separation between them. The case of very small obstacles/thin wires is expected to produce different interface conditions, analogous to those derived in \cite{ChapmanHewettTrefethen,HewettHewitt} in the 2D case. But we leave this case for future work.
\section{Statement of the problem}\label{}
\noindent 
Our objective is to derive effective interface conditions for electromagnetic scattering by a thin periodic layer of equispaced perfectly-conducting obstacles on the interface $\Gamma = \{\mathbf{x} = (x_1, x_2, x_3) \in \R^3 : x_3 =0\}$. 
Let  $\HatOmegaHole \in \R^3$  be the canonical  obstacle described by one of the following three cases (see Fig.~\ref{HatOmegaHolen}):
\begin{enumerate}
\item $\HatOmegaHole$ is a simply connected Lipschitz domain whose closure is contained in 
$(0,1)^2 \times (-\frac{1}{2}, \frac{1}{2})$.
\item 
$\HatOmegaHole = [0,1] \times (\frac{3}{8}, \frac{5}{8}) \times (-\frac{1}{8}, \frac{1}{8}) $, \ie~a wire (of square section) parallel to the direction $\be_1$.
\item 
$\HatOmegaHole = \{ [0,1] \times (\frac{3}{8}, \frac{5}{8}) \times (-\frac{1}{8}, \frac{1}{8}) \} \cup \{ (\frac{3}{8}, \frac{5}{8}) \times [0,1]\times (-\frac{1}{8}, \frac{1}{8})\}$, \ie~a cross-shape domain made of the union of two perpendicular wires (one parallel to $\be_1$ and the other parallel to $\be_2$).
\end{enumerate} 
\begin{figure}[htbp]
        \centering
        \begin{subfigure}[b]{0.3\textwidth}
        \begin{center}
                \includegraphics[width=0.8\textwidth, trim = 8cm 20cm 8cm 4.5cm, clip]{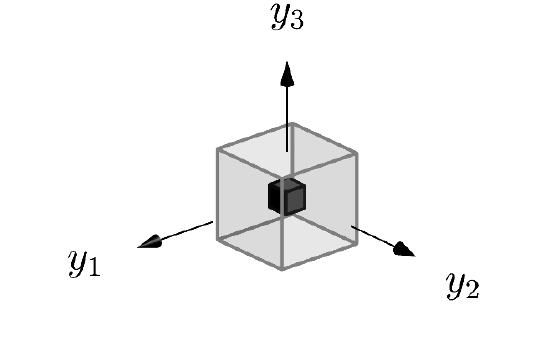}
                \caption{Case (i)}
                \end{center}
        \end{subfigure}%
        \begin{subfigure}[b]{0.3\textwidth}
        \begin{center}
                \includegraphics[width=0.8\textwidth,trim = 8cm 20cm 8cm 4.5cm, clip]{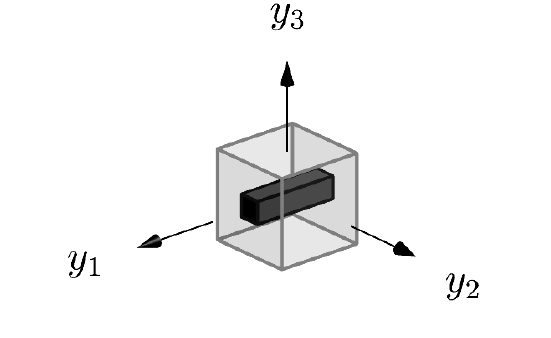}
                \caption{Case (ii)}
                \end{center}
        \end{subfigure}%
         \begin{subfigure}[b]{0.3\textwidth}
        \begin{center}
                \includegraphics[width=0.8\textwidth, trim = 8cm 20cm 8cm 4.5cm, clip]{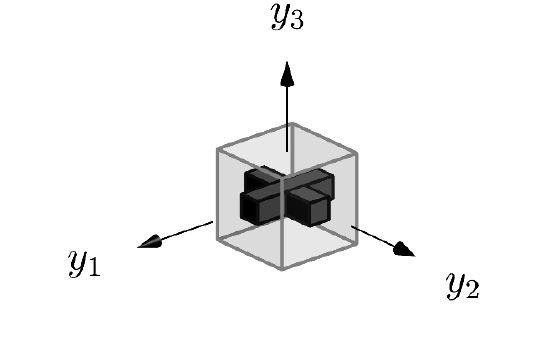}
                \caption{Case (iii)}
                \end{center}
        \end{subfigure}%
              \caption{The canonical obstacle $\HatOmegaHole$ in the three cases under consideration.}\label{HatOmegaHolen}
 \end{figure}
\noindent We construct the thin layer as a union of scaled and shifted versions of the canonical obstacle $\HatOmegaHole$. 
For $\delta>0$ we define $\mathscr{L}^{\delta}\subset \R^2\times[-\delta/2,\delta/2]$ by
\[
\mathscr{L}^{\delta} = 
{\rm int}\left( \bigcup_{(i,j) \in \Z^2}
\delta\left\lbrace \overline{\HatOmegaHole} + i \mathbf{e}_1 + j \mathbf{e}_2  \right\rbrace\right).
\] 
Our domain of interest is then $\Omega^\delta = 
\R^3
\setminus  \overline{\mathscr{L}^{\delta}}$ (cf. Fig~\ref{OmegaDelta}), and we define $\Gamma^\delta = \partial \Omega^\delta$.  
\begin{figure}[htbp]
        \centering
        \begin{subfigure}[b]{0.3\textwidth}
        \begin{center}
                \includegraphics[width=1\textwidth, trim = 3.5cm 10.5cm 2.5cm 5cm, clip]{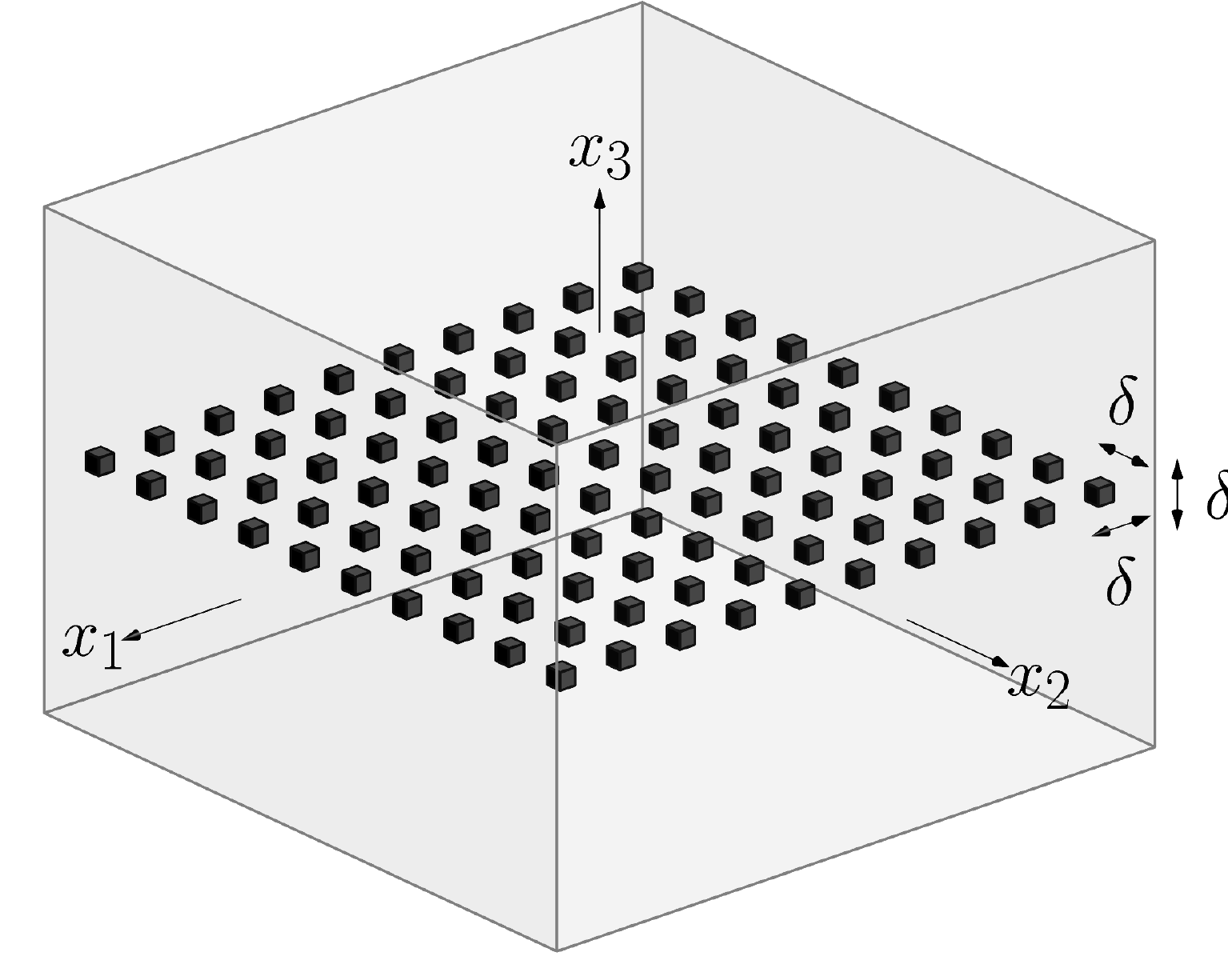}
                \caption{Case (i) - discrete obstacles}
                \end{center}
        \end{subfigure}%
        \begin{subfigure}[b]{0.3\textwidth}
        \begin{center}
                \includegraphics[width=1\textwidth,trim = 3.5cm 10.5cm 2.5cm 5cm, clip]{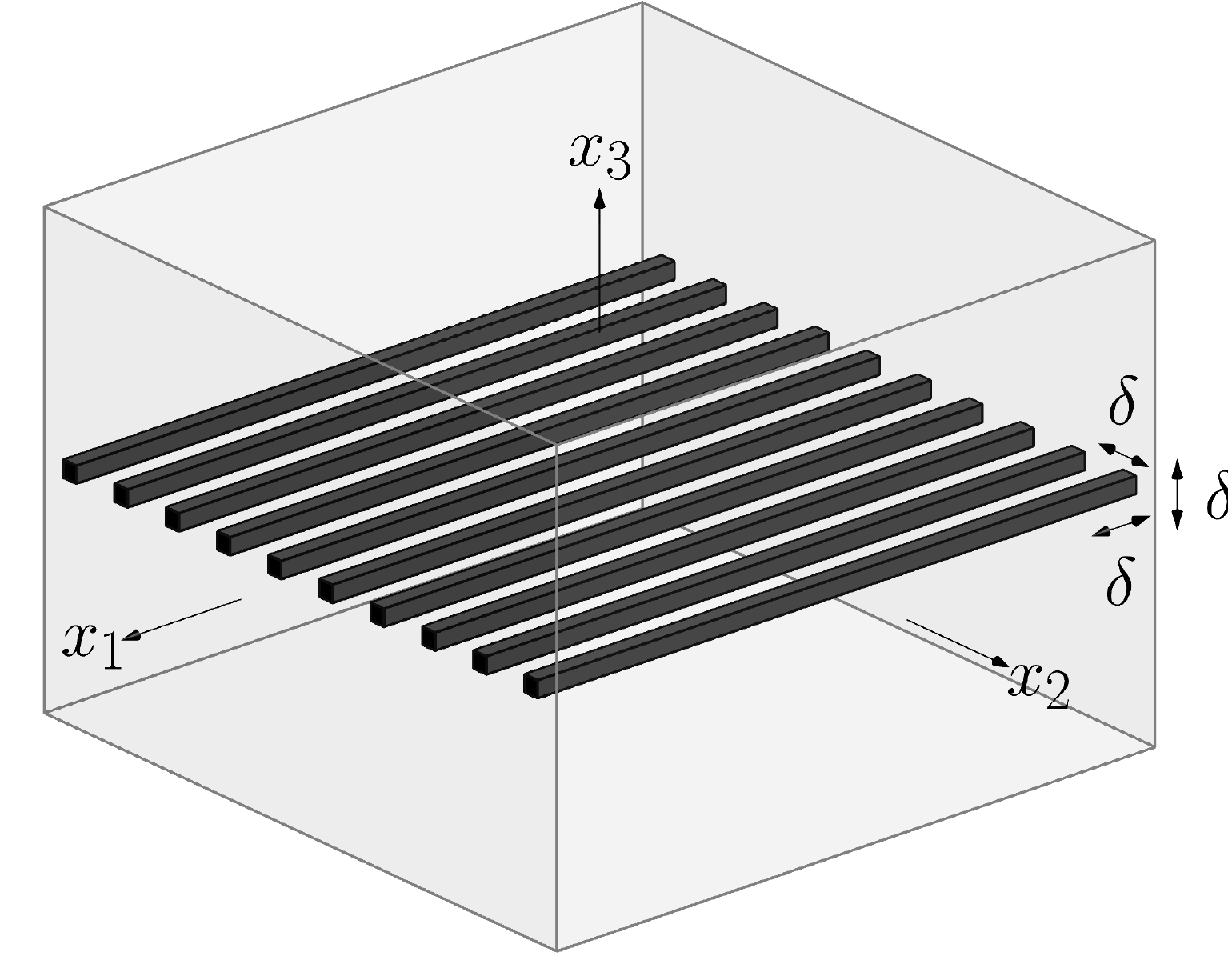}
                \caption{Case (ii) - parallel wires}
                \end{center}
        \end{subfigure}%
         \begin{subfigure}[b]{0.3\textwidth}
        \begin{center}
                \includegraphics[width=1\textwidth, trim = 3.5cm 10.5cm 2.5cm 5cm, clip]{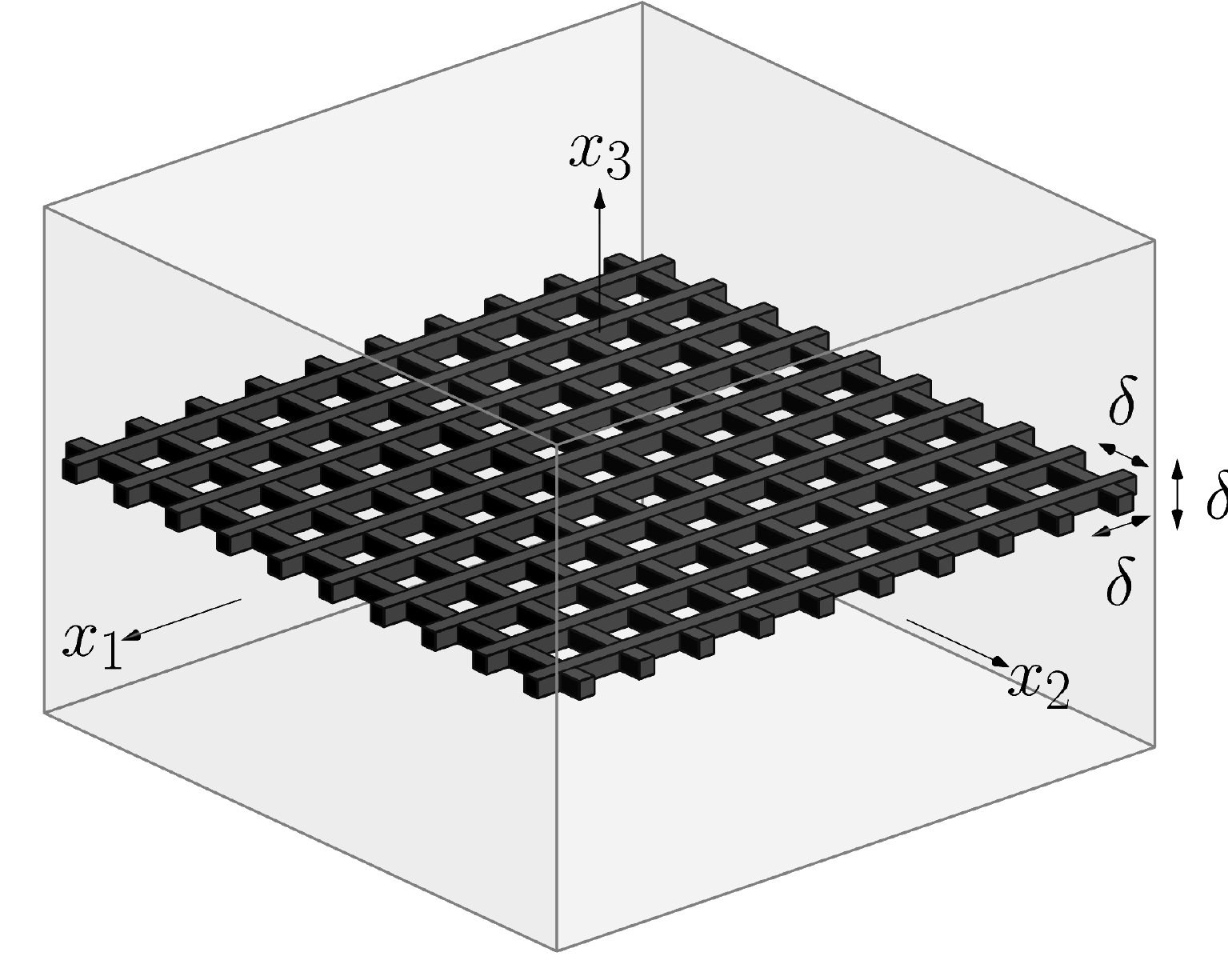}
                \caption{Case (iii) - wire mesh}
                \end{center}
        \end{subfigure}%
              \caption{The domain $\Omega^\delta$ in the three cases under consideration.}\label{OmegaDelta}
 \end{figure}
 
\noindent On the domain $\Omega^\delta$ we consider the solution $\bu^\delta$ of the Maxwell equations  
\begin{equation}\label{MaxwellEquation}
\rot \rot \bu^\delta - \omega^2 \varepsilon \bu^\delta = \mathbf{f} \quad  \mbox{in }  \Omega^\delta,
\end{equation}
where $\omega>0$ and $\varepsilon\in \mathbb{C}$, 
subject to the perfectly conducting boundary condition
\begin{equation}\label{BoundaryCondition}
\bu^\delta \times \bn = 0 \quad \mbox{on} \quad \Gamma^\delta.
\end{equation}
For analytical convenience we avoid any complications arising from far-field behaviour by assuming that 
${\rm Re}[\varepsilon] > 0$ and ${\rm Im}[\varepsilon] > 0$. 
The assumption that ${\rm Im}[\varepsilon] > 0$ 
could be relaxed to ${\rm Im}[\varepsilon] \geq 0$ at the expense of some technical modifications, including the imposition of an appropriate radiation condition.
We also assume that the support of $\mathbf{f}$ does not intersect the interface $\Gamma$. 
Then, given $\mathbf{f} \in \left(L^2(\Omega^\delta)\right)^3$, the Lax-Milgram Lemma ensures that Problem~(\ref{MaxwellEquation})-(\ref{BoundaryCondition}) has a unique solution $\bu^\delta$ in the standard function space 
\begin{equation}\label{DefinitionHrot}
H(\rot ;  \Omega^\delta) = \left\{\bv \in (L^2(\Omega^\delta))^3: \rot \bv \in (L^2(\Omega^\delta))^3 \right\},
\end{equation} 
equipped with the usual graph norm
$
\| \bv  \|_{H(\rot ;  \Omega^\delta)} = \big( \| \bv  \|_{\left(L^2(\Omega^\delta)\right)^3}^2 + \| \rot \bv \|_{\left(L^2(\Omega^\delta)\right)^3}^2 \big)^{1/2}$.
Moreover, one can prove that there exists $C>0$, independent of $\delta$, such that
\begin{equation}\label{Stability}
\| \bu^\delta \|_{H(\rot ;  \Omega^\delta)} \leq C \| \mathbf{f}  \|_{\left(L^2(\Omega^\delta)\right)^3}, \qquad \text{for all }0<\delta<1.
\end{equation} 
The objective of this work is to identify formally the limit $\bu_0$ of $\bu^\delta$ as $\delta$ tends to $0$.  This limit solution is defined in the 
union of two distinct domains $\Omega^\pm = \{\bx\in\R^3 :\pm x_3>0\}$, whose common interface is $\Gamma$. Our main result is the following:
\begin{theorem}\label{Prop1} The limit solution  $\bu^0$ satisfies the Maxwell equations
\begin{equation}\label{EquationVolumeu0}
\rot \rot \bu^0 - \omega^2 \varepsilon \bu^0 = \mathbf{f} \quad \mbox{in }  \Omega^+ \cup \Omega^-,
\end{equation} 
together with the following interface conditions on $\Gamma$:
\begin{enumerate}
\item[] Case (i): $[ \bu_0 \times \be_3 ]_{\Gamma} = \boldsymbol{0}$ and $[\rot \bu_0  \times \be_3]_{\Gamma}=0$.
\item[] Case (ii): $\bu_0 \cdot \be_1 = 0$ on $\Gamma$, $[ \bu_0 \cdot \be_2]_{\Gamma}=0$, and $[ (\rot \bu_0 \times \be_3)\cdot \be_2 ]_{\Gamma} = 0$.
\item[] Case (iii): $\bu_0 \times \be_3=\boldsymbol{0}$ on $\Gamma$.
\end{enumerate} 
\end{theorem} 
Let us make a few comments on this result. First, we emphasize that the nature of the limit problem depends strongly on the topology of the thin layer of obstacles $\mathscr{L}^{\delta}$. In case (iii), where $\mathscr{L}^{\delta}$ comprises a wire mesh, we observe the ``Faraday cage effect'', where the effective interface $\Gamma$ is a solid perfectly conducting sheet. Hence if the support of $\mathbf{f}$ lies in $ \Omega^+$ (above the layer $\mathscr{L}^{\delta}$), then $\bu_0=\mathbf{0}$ in $\Omega_-$. 
In other words, despite the holes in its structure, the layer $\mathscr{L}^{\delta}$ shields the domain $\Omega^-$ from electromagnetic waves of all polarizations. 
At the opposite extreme, in case (i), where $\mathscr{L}^{\delta}$ comprises discrete obstacles, the interface is transparent and there is no shielding effect. In the intermediate case (ii), where $\mathscr{L}^{\delta}$ comprises an array of parallel wires, one observes polarization-dependent shielding. Fields polarized parallel to the wire axis are shielded, whereas those polarized perpendicular to the wire axis are not. 
Note that this case (ii) includes as a subcase the simpler two-dimensional situation studied in \cite{HewettHewitt,Holloway,MarigoMaurel} where the fields are invariant in the direction of the wire axis. 

\noindent The remainder of this note is dedicated to the formal proof of Theorem~\ref{Prop1}. The proof is based on the construction of an asymptotic expansion of $\bu^\delta$ using the method of matched asymptotic expansions (cf.~\cite{MazyaNazarovPlam}). To simplify the computation, we work with the first order formulation of~\eqref{MaxwellEquation}, introducing the magnetic field $\bh^\delta =  \frac{1}{\ri \omega} \rot \bu^\delta$ (see e.g.\ \cite{Monk}) and obtaining
\begin{equation}\label{MaxwellOrdre1}
\begin{cases}
- \ri \omega \bh^\delta  + \rot \bu^\delta = 0  & \mbox{in }  \Omega^\delta,\\ 
 - \ri \omega \bu^\delta  - \rot \bh^\delta = -\frac{1}{\ri \omega} \mathbf{f}   & \mbox{in }  \Omega^\delta,
 \end{cases} \quad \bu^\delta  \times \bn = 0  \mbox{ and }  \bh^\delta\cdot \bn = 0\mbox{ on } \Gamma^\delta. 
\end{equation} 
Far from the periodic layer $\mathscr{L}^{\delta}$, we construct an expansion of $\bh^\delta$ and $\bu^\delta$ of the form
 \begin{equation}\label{FFExpansion}
 \bh^\delta =  \bh_0(\bx) + \delta \bh_1(\bx) + \cdots,  \quad   \bu^\delta =  \bu_0(\bx) + \delta \bu_1(\bx) + \cdots, \quad \bx = (x_1, x_2, x_3), 
 \end{equation}
 and, in the vicinity of $\mathscr{L}^{\delta}$,
 \begin{equation}\label{NFExpansion}
  \bh^\delta =  \bH_0(x_1, x_2, \frac{\bx}{\delta}) + \delta \bH_1(x_1, x_2, \frac{\bx}{\delta}) + \cdots,  \quad   \bu^\delta =  \bU_0(x_1, x_2, \frac{\bx}{\delta})+ \delta \bU_1(x_1, x_2, \frac{\bx}{\delta}) + \cdots, \quad 
 \end{equation}
 where, for $i \in \{0,1\}$, $\bH_i(x_1, x_2,y_1, y_2, y_3)$ and $\bU_i(x_1, x_2,y_1, y_2, y_3)$ are assumed to be $1$-periodic in both $y_1$ and $y_2$.  Near and far field expansions communicate through so-called matching conditions, which ensure that the far and near field expansions coincide in some intermediate areas. Since we are only interested in the leading order terms, it is sufficient to consider only the $O(1)$ matching conditions, namely
 \begin{equation}\label{MatchingConditionOrdre0}
 \lim_{x_3 \rightarrow 0^\pm} \bh_0 = \lim_{y_3 \rightarrow \pm \infty}\bH_0 \quad \mbox{and } \quad \lim_{x_3 \rightarrow 0^\pm} \bu_0 = \lim_{y_3 \rightarrow \pm \infty}\bU_0.
 \end{equation} 
 Inserting~\eqref{FFExpansion} into~\eqref{MaxwellEquation} and separating the different powers of $\delta$ directly gives~\eqref{EquationVolumeu0}. To obtain the interface conditions, we have to study the problems satisfied by $\bU_0$ and $\bH_0$:
 \begin{equation}\label{NearFieldProblemOrder0}
\begin{cases}
\rot_y \bU_0  = 0  & \mbox{in } \mathscr{B}_\infty, \\
 \dive_y \bU_0 =  0 & \mbox{in } \mathscr{B}_\infty, \\
 \bU_0 \times \bn =  0 & \mbox{on } \partial \mathscr{B}_\infty, \\
 \end{cases} \quad  \begin{cases}
\rot_y \bH_0  = 0  & \mbox{in } \mathscr{B}_\infty, \\
 \dive_y \bH_0 =  0 & \mbox{in } \mathscr{B}_\infty, \\
 \bH_0 \cdot \bn =  0 & \mbox{on } \partial \mathscr{B}_\infty,\\
 \end{cases}
 \mathscr{B}_\infty =  \Omega^1 = \R^3\setminus \overline{\mathscr{L}^1} .
 \end{equation}  
 \section{The spaces $K_N(\mathscr{B}_\infty)$ and $K_T(\mathscr{B}_\infty)$} 
\noindent 
Denoting by $\mathscr{B}$ the restriction of $\mathscr{B}_\infty$ to the strip $(0, 1)^2 \times (-\infty,\infty)$, 
we introduce the 
spaces
\begin{multline}
\mathscr{H}_{N}(\mathscr{B}_\infty) = \left\{  
\bu \in H_{\rm loc}({\rot};\mathscr{B}_\infty) \cap H_{\rm loc}(\dive;\mathscr{B}_\infty):  
\; \bu \mbox{ is }  1\mbox{-periodic in } y_1  \mbox{ and }  y_2,  \right. 
\\  \frac{\bu_{|\mathscr{B}}}{\sqrt{1+(y_3)^2}} \in (L^2(\mathscr{B}))^3, \quad   \rot \bu_{|\mathscr{B}}  \in (L^2(\mathscr{B}))^3,  \quad \xdiv \bu_{|\mathscr{B}} \in L^2(\mathscr{B}),    \quad \bu \times \bn = 0 \; \mbox{on} \; \partial \mathscr{B}_\infty  \big\} ,
\end{multline}
\begin{multline}
\mathscr{H}_{T}(\mathscr{B}_\infty) = \left\{  
\bh \in H_{\rm loc}({\rot};\mathscr{B}_\infty) \cap H_{\rm loc}(\dive;\mathscr{B}_\infty): 
\;\bh \mbox{ is 1-periodic in } y_1  \mbox{ and } y_2 , \right. 
\\  \frac{\bh_{|\mathscr{B}}}{\sqrt{1+(y_3)^2}} \in (L^2(\mathscr{B}))^3, \quad   \rot \bh_{|\mathscr{B}}  \in (L^2(\mathscr{B}))^3,  \quad \xdiv \bh_{|\mathscr{B}} \in L^2(\mathscr{B}),    \quad \bh \cdot \bn = 0 \; \mbox{on} \; 
\partial \mathscr{B}_\infty
\big\},
\end{multline}
both of which include periodic vector fields in $H_{\rm loc}({\rot};\mathscr{B}_\infty) \cap H_{\rm loc}(\dive;\mathscr{B}_\infty) $ that tend to a constant vector as $|y_3|\to\infty$. Investigation of \eqref{NearFieldProblemOrder0} requires the characterization of the so-called normal and tangential cohomology spaces $K_N(\mathscr{B}_\infty)$ and $K_T(\mathscr{B}_\infty)$ defined by (see \cite{DaugeBernardiAmrouche})
\begin{equation}
 K_N(\mathscr{B}_\infty) = \left\{ \bu \in \mathscr{H}_{N}(\mathscr{B}_\infty) ,  \rot \bu = 0, \dive \bu =0 \right\},\; K_T(\mathscr{B}_\infty) = \left\{ \bh \in \mathscr{H}_{T}(\mathscr{B}_\infty) ,  \rot \bh = 0, \dive \bh =0 \right\}.  
\end{equation}  
This characterization involves the representation of elements of $K_N(\mathscr{B}_\infty)$ and $K_T(\mathscr{B}_\infty)$ as gradients of harmonic scalar potentials, constructed by solving certain variational problems in the space
\begin{equation}\label{W1}
 \mathcal{W}_{1}(\mathscr{B_\infty}) = \big\{ p \in H^1_{\rm loc}(\mathscr{B}_{\infty}): \;p \mbox{ is 1-periodic in } y_1 \mbox{ and } y_2, \frac{p_{|\mathscr{B}}} {\sqrt{1+(y_3)^2}} \in L^2(\mathscr{B}),   \nabla p_{|\mathscr{B}}  \in L^2(\mathscr{B})  \big\},
\end{equation}
and variants of it. 
In each case the existence and uniqueness of the potential follows 
from the Lax-Milgram Lemma. 
While we do not reproduce the proofs here, we remark that the unbounded nature of the domain $\mathscr{B}$ requires us, when verifying coercivity of the requisite bilinear forms, to appeal to the inequality
\begin{align}
\label{}
\left\|\frac{p}{\sqrt{1+(y_3)^2}}\right\|_{L^2(\mathscr{B_+})} \leq 2 \| \nabla p\|_{L^2(\mathscr{B_+})}, \qquad \mathscr{B_+} = \mathscr{B}\cap\{y_3>0\},
\end{align}
valid if $p\in C^\infty(\overline{\mathscr{B}_+})$, $p/\sqrt{1+(y_3)^2} \in L^2(\mathscr{B_+})$, $\nabla p\in L^2(\mathscr{B_+})$ and $p=0$ in a neighbourhood of $\{y_3=0\}$, 
which is an elementary consequence of the Hardy inequality~\cite[Lemma 2.5.7]{Nedelec}
\begin{align}
\label{}
\int_0^\infty t^{-2}|\varphi(t)|^2\,{\rm d} t \leq 4 \int_0^\infty |\varphi'(t)|^2\,{\rm d} t, \qquad \varphi\in C_0^\infty ((0,\infty)).
\end{align}

\subsection{Characterization of $K_N(\mathscr{B}_\infty)$}
\noindent To characterize $K_N(\mathscr{B}_\infty)$ we first define two functions $p^\pm_{3} \in  H^1_{\rm loc}(\mathscr{B}_\infty)$, $1$-periodic in $y_1$ and $y_2$, 
such that 
$$
\begin{cases}
- \Delta p^\pm_3 = 0 &   \mbox{in }\mathscr{B}_\infty,\\
p^\pm_3 = 0 & \mbox{on } \partial \mathscr{B}_\infty, \\
\end{cases}
\quad \lim_{y_3 \rightarrow \pm \infty} \nabla p^\pm_3  =  \be_3,   \quad
\lim_{y_3 \rightarrow \mp \infty} \nabla p^\pm_3  = 0.
$$ 
Then, in case (i) we introduce the functions $\tilde{p}_1\in \mathcal{W}_{1}(\mathscr{B_\infty})$ and $p_1 \in  H^1_{\rm loc}(\mathscr{B}_\infty)$,  such that 
$$
\begin{cases}
- \Delta \tilde{p}_1 = 0 &   \mbox{in }\mathscr{B}_\infty,\\
\tilde{p}_1 =  - \mathscr{P} \mathscr{R} y_1 & \mbox{on }  \partial \mathscr{B}_\infty, \\ 
\end{cases} \quad \mbox{and} \quad p_1 = \tilde{p}_1 + y_1.
$$ 
 Here, for any function $u \in L^2_{\rm loc}(\mathscr{B}_\infty)$, $\mathscr{R} u$ denotes its restriction to $\mathscr{B}$, while for any function $u \in  L^2_{\rm loc}(\mathscr{B})$, $\mathscr{P}u$ denotes its periodic extension to  $\mathscr{B}_\infty$. 
Similarly, in cases (i) and (ii) we introduce the functions $\tilde{p}_2\in \mathcal{W}_{1}(\mathscr{B_\infty})$ and $p_2 \in  H^1_{\rm loc}(\mathscr{B}_\infty)$,  such that 
$$
\begin{cases}
- \Delta \tilde{p}_2 = 0 &   \mbox{in }\mathscr{B}_\infty,\\
\tilde{p}_2 =  - \mathscr{P} \mathscr{R} y_2 & \mbox{on }  \partial \mathscr{B}_\infty, \\ 
\end{cases} \quad \mbox{and} \quad p_2= \tilde{p}_2 + y_2.
$$ 
We emphasize that it is not possible to construct $\tilde{p}_1$ in cases (ii) and (iii), and it is not possible to construct $\tilde{p}_2$ in case (iii).  An adaptation of the proof of \cite[Proposition 3.18]{DaugeBernardiAmrouche} leads to the following result:
\begin{proposition}\label{PropositionKN}~\\[-3ex]
 \begin{enumerate}
 \item[]{Case (i)}:  $K_N$ is the space of dimension $4$  given by $K_N(\mathscr{B}_\infty)  = span \left\{   \nabla  {p}_1,  \nabla  {p}_2 ,\nabla {p}^-_3,  \nabla p^+_3  \right\}.$ 
 \item[]{Case (ii)}: $K_N$ is the space of dimension $3$ given by
$
 K_N(\mathscr{B}_\infty)=  span \left\{   \nabla{p}_2, \nabla {p}^-_3,  \nabla p^+_3 \right\}.
 $
 \item[] {Case (iii)}: $K_N$ is the space of dimension $2$ given by
$
  K_N(\mathscr{B}_\infty)=  span \left\{   \nabla {p}^-_3,  \nabla p^+_3 \right\}.$
\end{enumerate} 
\end{proposition}  
\begin{proof}[Sketch of the proof in case (ii)]   First, one can verify directly that the family $\left\{   \nabla{p}_2, \nabla {p}^-_3,  \nabla p^+_3\right\}$ is linearly independent (using the limit of $\nabla{p}_2$ and $\nabla {p}^\pm_3$  as $y_3$ tends to $\pm\infty$). Moreover, it is clear that $\nabla{p}_2$ and $\nabla {p}^\pm_3$ belong to $K_N(\mathscr{B}_\infty)$. Now, let $\bu \in  K_N(\mathscr{B}_\infty)$. Since $\mathscr{B}_\infty$ is connected, there exists $p \in H^1_{\rm loc}(\mathscr{B}_\infty)$, unique up to the addition of a constant, such that $\bu = \nabla p$. (This follows e.g.\ from applying \cite[Theorem 3.37]{Monk} on an increasing sequence of nested subsets of $\mathscr{B}_\infty$ after extension of $\mathbf{u}$ by zero inside $\R^3\setminus\overline{\mathscr{B}_\infty}$.) Moreover, $\nabla p$ is periodic and there exists a real sequence $(c_j)_{j\in \Z}$ such that
$$
-\Delta p =0 \;\mbox{in } \mathscr{B}_\infty, \quad p = c_j \; \mbox{on } \partial \mathscr{B}_{\infty,j}  =  \partial \mathscr{B}_\infty \cap \{ j < y_2 < (j+1) \}.
$$  
Because $\nabla p$ is periodic and $\frac{\bu_{|\mathscr{B}}}{\sqrt{1+(y_3)^2}} \in (L^2(\mathscr{B}))^3$, there exists four constants $\alpha_1$, $\alpha_2$, $\alpha_3^\pm$ such that $$\tilde{p} = p - \alpha_1 y_1 - \alpha_2 y_2 - \sum_{\pm} \alpha_3^\pm p^\pm_3 \in  \mathcal{W}_1( \mathscr{B}_\infty).$$
Since 
$\tilde{p}  = c_j  - \alpha_1 y_1 - \alpha_2 y_2$ on $\partial \mathscr{B}_{\infty,j}$, the periodicity of $\tilde{p}$ in $y_1$ implies that $\alpha_1 = 0$, while its periodicity in $y_2$ leads to $c_j = c_0 + \alpha_2 j$.  As a result, 
$$
\tilde{p}  = c_0 - \alpha_2 (y_2 -j) \quad \; \mbox{on } \partial \mathscr{B}_{\infty,j}.
$$
Since $\tilde{p}$ is harmonic, we deduce that $\tilde{p} = c_0 + \alpha_2 \tilde{p}_2$, and hence that $p =c_0+ \alpha_2 p_2 + \sum_{\pm} \alpha_3^\pm p^\pm_3$, which completes the proof. 
Cases (i) and (iii) follow similarly.
\end{proof}
\subsection{Characterization of $K_T(\mathscr{B}_\infty)$}
\noindent First, let us define $q_3 \in H^1_{\rm loc}(\mathscr{B}_{\infty})$ as the unique function such that 
$$
\begin{cases}
- \Delta q_3 = 0 &   \mbox{in }\mathscr{B}_\infty,\\
\partial_{\bn} q_3= 0 & \mbox{on } \partial\mathscr{B}_\infty, \\
\end{cases}
\quad \lim_{y_3 \rightarrow \pm \infty} \nabla q_3  =  \be_3, \quad \lim_{y_3 \rightarrow + \infty} q_3 - y_3  = 0.
$$ 
Then for $i \in \{1, 2\}$ we introduce the functions $\tilde{q}_i \in \mathcal{W}_{1}(\mathscr{B_\infty})$ and $q_i \in  H^1_{\rm loc}(\mathscr{B}_\infty)$ such that 
$$
\begin{cases}
- \Delta \tilde{q_i} = 0 &   \mbox{in }\mathscr{B}_\infty,\\
\partial_{\bn} \tilde{q_i}=  - \be_i \cdot \bn & \mbox{on } \partial\mathscr{B}_\infty, 
\end{cases}
\quad \lim_{y_3 \rightarrow + \infty} \tilde{q_i}  = 0,  \quad \mbox{and} \quad
q_i = \tilde{q}_i + y_i.
$$ 
In case (ii) we introduce a set of `cuts' $\Sigma$ defined by 
$$ 
\Sigma = \bigcup_{j \in \Z} \Sigma_j,  \quad \mbox{where} \quad \Sigma_j = \Sigma_0 + j \be_2, \quad 
\Sigma_0  = (-\infty,\infty)\times(-\tfrac{3}{8},\tfrac{3}{8})\times\{0\}.
$$ 
Similarly, in case (iii) we introduce the cuts
$$
\Sigma = \bigcup_{(i,j) \in \Z^2} \Sigma_{ij},  \quad \mbox{where} \quad \Sigma_{ij} =   \Sigma_{0 0} + i \be_1 +  j \be_2,   \quad 
\Sigma_{0 0}= (-\tfrac{3}{8},\tfrac{3}{8})^2\times\{0\}.
$$ 
In both cases, $\mathscr{B}_\infty \setminus \Sigma$ is then the union of the two simply connected domains $
\mathscr{B}_\infty^\pm = \left( \mathscr{B}_\infty \setminus \Sigma \right) \cap \{ \pm y_3 > 0\} $.
We denote by $\mathcal{W}_1(\mathscr{B}_\infty^\pm)$ the space defined by formula~\eqref{W1} replacing $\mathscr{B}_\infty$ with $\mathscr{B}_\infty^\pm$.
In case (ii) we let  $ \qtwoperpm= \left( ( \qtwoperpm)_+ ,( \qtwoperpm)_-\right) \in \mathcal{W}_1(\mathscr{B}_\infty^+) \times  \mathcal{W}_1(\mathscr{B}_\infty^-)$ be the unique solutions to \begin{equation}\label{p2plus}
\begin{cases}
- \Delta \qtwoperpm = 0  & \mbox{in } \mathscr{B}_\infty\setminus \Sigma  , \\
\partial_\bn  \qtwoperpm = -\be_2 \cdot \bn  & \mbox{on } \partial \mathscr{B}_\infty^\pm \cap \partial \mathscr{B}_\infty ,\\
\partial_\bn  \qtwoperpm = 0  & \mbox{on } \partial \mathscr{B}_\infty^\mp\cap \partial \mathscr{B}_\infty,
\end{cases}\quad 
\begin{cases}
 [ \qtwoperpm ]_{\Sigma_j} =  \pm (j-y_2), \\
[ \partial_{y_3}   \qtwoperpm]_{\Sigma_j} = 0 ,
\end{cases}
\lim_{y_3 \rightarrow +\infty}   \qtwoperpm =0 , 
 \quad \quad \end{equation} 
and we define $q^\pm_2 = \qtwoperpm + y_2 1_{ \mathscr{B}_\infty^\pm}$,
 $1_{ \mathscr{B}_\infty^\pm}$ being the indicator function of $\mathscr{B}_\infty^\pm$. In case (iii) the functions $q^\pm_2$ are defined similarly, except that we replace $\Sigma_{j}$ by $\Sigma_{ij}$ in the jump conditions. 
In case (iii) we additionally introduce the functions $ \qoneperpm= \left( (  \qoneperpm)_+ ,(  \qoneperpm)_-\right) \in \mathcal{W}_1(\mathscr{B}_\infty^+) \times  \mathcal{W}_1(\mathscr{B}_\infty^-)$ as the unique solutions to
\begin{equation}\label{p2plus}
\begin{cases}
- \Delta \qoneperpm= 0  & \mbox{in } \mathscr{B}_\infty\setminus \Sigma  , \\
\partial_\bn  \qoneperpm = -\be_1 \cdot \bn  & \mbox{on } \partial \mathscr{B}_\infty^\pm\cap \partial \mathscr{B}_\infty,\\
\partial_\bn  \qoneperpm= 0  & \mbox{on } \partial \mathscr{B}_\infty^\mp\cap \partial \mathscr{B}_\infty,
\end{cases} \;\;
\begin{cases}
[ \qoneperpm ]_{\Sigma_{ij}} =  \pm (i-y_1), \\
[ \partial_{y_3}   \qoneperpm]_{\Sigma_{ij}} = 0, \\
\end{cases} \;
\lim_{y_3 \rightarrow +\infty}   \qoneperpm =0, 
\end{equation} 
and we define $q^\pm_1 =  \qoneperpm + y_1 1_{ \mathscr{B}_\infty^\pm}$. Then, adapting the proof of \cite[Proposition 3.14]{DaugeBernardiAmrouche} one obtains the following result:
 \begin{proposition}\label{PropositionKT}~\\[-3ex]
  \begin{enumerate}
 \item[]{Case (i)}:  $K_T$ is the space of dimension $3$ given by
 $
 K_T(\mathscr{B}_\infty)  = span \left\{\nabla  {q}_1,  \nabla  {q}_2 ,  \nabla {q}_{3}  \right\}.
 $ 
 \item[]{Case (ii)}: $K_T$ is the space of dimension $4$ given by
$
 K_T(\mathscr{B}_\infty)=  span \left\{ \nabla  {q}_1,  \nabla  {q}^+_2 ,  \nabla  {q}^-_2, \nabla {q}_{3}\right\}.
 $
 \item[] {Case (iii)}: $K_T$ is the space of dimension $5$ given by
$
  K_T(\mathscr{B}_\infty)=  span \left\{  \nabla  {q}^+_1 ,  \nabla  {q}^-_1, \nabla  {q}^+_2 ,  \nabla  {q}^-_2, \nabla {q}_{3}\right\}.
 $
\end{enumerate} 
\end{proposition}
\begin{proof}[Sketch of the proof in case (ii)]
As in the proof of Proposition~\ref{PropositionKN}, it is not difficult to prove that the family $\{ \nabla q_1, \nabla q_2^+, \nabla q_2^-, \nabla q_3\}$ is linearly independent and that its elements belong to  $K_T(\mathscr{B}_\infty)$. Then,
let $\bh = (\bh_+, \bh_-) \in K_T(\mathscr{B}_\infty)$. Since $\mathscr{B}_\infty^\pm$ are simply connected, there exists $q = (q_+, q_-) \in H^1_{\rm loc}(\mathscr{B}_\infty^+)\times H^1_{\rm loc}(\mathscr{B}_\infty^-)$ and a real sequence $(c_j)_{j\in \Z}$ such that 
$$
\bh_\pm = \nabla q_\pm, 
\quad \Delta q = 0   \mbox{ in } \mathscr{B}_\infty\setminus \Sigma, \quad  \partial_\bn q = 0 \mbox{ on } \partial  \mathscr{B}_\infty,
\quad [ q ]_{\Sigma_j} = c_j, \quad [\partial_{y_3}q ]_{\Sigma_j} =0.
$$ 
Since $\nabla q$ is periodic and $\frac{\bh_{|\mathscr{B}}}{\sqrt{1+(y_3)^2}} \in (L^2(\mathscr{B}))^3$, there exist five constants  $\alpha_1^\pm$, $\alpha_2^\pm$ and $\alpha_3$ such that
$$
\tilde{q}_\pm= q_\pm- \alpha_1^\pm y_1 -  \alpha_2^\pm y_2  - \alpha_3 q_3 \in  \mathcal{W}_1( \mathscr{B}_\infty^\pm).
$$
Because $\tilde{q}=(\tilde{q}_+,\tilde{q}_-)$ satisfies $[ \tilde{q}]_{\Sigma_j} = c_j - [\alpha_1] y_1 - [\alpha_2] y_2$ and $\tilde{q}$ is periodic, we find that $[\alpha_1]=0$, and $
c_j = c_0 +  j [\alpha_2]$  for each $j \in \Z$. To conclude, it suffices to note that $\hat{q} = q - \alpha_1 q_1  - \sum_{\pm} \alpha_2^\pm q_{2, \pm} - \alpha_3 q_3$ is periodic and satisfies
$$
\begin{cases}
- \Delta \hat{q}= 0 & \mbox{in }  \mathscr{B}_\infty \setminus \Sigma, \\
\partial_\bn \hat{q} = 0 & \mbox{in }  \partial \mathscr{B}_\infty, \\
\end{cases} \quad
\begin{cases} 
[ \hat{q}]_{\Sigma_j} = c_0,    \\
[ \partial_\bn  \hat{q}]_{\Sigma_j} =0, \\
\end{cases} \quad
\int_{\Sigma_j} \partial_{y_3}  \hat{q}  = 0, 
$$
which proves that $\hat{q}$ is constant in each of $\mathscr{B}_\infty^\pm$.
\end{proof}
 \section{Formal proof of Theorem~\ref{Prop1}}
 \noindent We treat the three cases separately. In case (i), using  Propositions~\ref{PropositionKN}-\ref{PropositionKT}, we have
 $$
 \bU_0 = \sum_{i=1}^2 a_i(x_1, x_2) \nabla p_i + \sum_{\pm}  a_3^\pm \nabla p^\pm_3  \quad \mbox{and} \quad \bH_0 = \sum_{i=1}^3 b_i(x_1, x_2) \nabla q_i. 
 $$ 
 The behaviour at infinity of the functions $p_i$ and $q_i$ and the matching conditions~\eqref{MatchingConditionOrdre0} then imply
 $$
 a_i = (\bu_0)_i^\pm(x_1, x_2,0) \quad b_i  = (\bh_0)_i^\pm(x_1, x_2,0)  \quad \forall  i \in \{ 1, 2 \},
 $$ 
and, consequently (by~\eqref{MaxwellOrdre1}), that 
$[\bu_0\times \be_3]_\Gamma = 0$ 
and $[ \rot \bu_0 \times \be_3]_\Gamma=0$. In case (ii) we have 
$$
\bU_0 = a_2(x_1, x_2) \nabla p_2 + \sum_{\pm}  a_3^\pm \nabla p^\pm_3  \quad \mbox{and} \quad \bH_0 = b_1(x_1, x_2) \nabla q_1  + \sum_{\pm} b_2^\pm(x_1, x_2) \nabla q_2^\pm +b_3(x_1, x_2) \nabla q_3,
$$ 
which, together with the matching conditions~\eqref{MatchingConditionOrdre0}, leads to
$
(\bu_0)_1^\pm(x_1,x_2) = 0$,  $[(\bu_0)_2]_\Gamma = 0$, $[(\bh_0)_1]_\Gamma =0$.
Finally, in case (iii) we have
$
\bU_0  = \sum_{\pm}  a_3^\pm \nabla p^\pm_3$,
which implies that $(\bu_0)_i^\pm(x_1, x_2) =0$ for $i =1 $ or $2$. 

\smallskip
\begin{rem}
We point out that our formal proof can be made rigorous by justifying the asymptotic expansions~\eqref{FFExpansion}-\eqref{NFExpansion}. This can be done a posteriori by constructing an approximation of $\bu^\delta$ on $\Omega^\delta$  (based on the truncated series~\eqref{FFExpansion}-\eqref{NFExpansion}) and using the stability estimate~\eqref{Stability} (see \cite{MazyaNazarovPlam}).  However, this would require us to identify the terms of order $1$ in the expansions, which is beyond the scope of this note.
\end{rem}


\begin{thebibliography}{00}


\bibitem{Faraday}Faraday M., \textit{Experimental researches in electricity}, vol. 1, secs. 1173--4 (reprinted from
Philosophical Transactions of 1831--1838), Richard and John Edward Taylor, London, 1839.
(http://www.gutenberg.org/ebooks/14986)
\bibitem{ChapmanHewettTrefethen}Chapman S.J., Hewett D.P. and Trefethen L.N., \textit{Mathematics of the Faraday cage}, Siam Review, 57(3), 398--417, 2015.
\bibitem{HewettHewitt}Hewett D.P., Hewitt I.J., \textit{Homogenized boundary conditions and resonance effects in Faraday cages}, Proc. R. Soc. A, 472(2189), 20160062, 2016.
\bibitem{Holloway} Holloway C.L., Kuester E.F., Dienstfrey, A., \textit{A homogenization technique for obtaining generalized sheet transition conditions for an arbitrarily shaped coated wire grating}, Radio Sci., 49(10), 813--850, 2014.
\bibitem{MarigoMaurel}Marigo J.J., Maurel A., \textit{Two-scale homogenization to determine effective parameters of thin metallic-structured films}, Proc. R. Soc. A, 472(2192), 20160068, 2016.
\bibitem{DelourmeWellPosedMax}Delourme B., Haddar H., Joly P., \textit{On the well-posedness, stability and accuracy of an asymptotic model for thin periodic interfaces in electromagnetic scattering problems},  Math. Method Appl. Sci, 23(13), 2433--2464, 2013.
\bibitem{DelourmeHighOrderMax}Delourme B., \textit{High-order asymptotics for the electromagnetic scattering by thin periodic layers}, Math. Method Appl. Sci, 38(5), 811--833, 2015.
\bibitem{SchweizerUrban}
Schweizer B., Urban M., \textit{Effective Maxwell's equations in general periodic microstructures}, Appl. Anal., published online - doi.org/10.1080/00036811.2017.1359563, 2017.
\bibitem{MazyaNazarovPlam}Maz'ya V., Nazarov S., Plamenevskij B., \textit{Asymptotic theory of elliptic boundary value problems in singularly perturbed domains}, Birkh\"auser, 2012.
\bibitem{Monk}Monk P., \textit{Finite element methods for Maxwell's equations}, Oxford University Press, 2003.
\bibitem{DaugeBernardiAmrouche}Amrouche C., Bernardi C., Dauge M., Girault V., \textit{Vector potentials in three-dimensional non-smooth domains.}, Math. Method Appl. Sci., 21(9), 823--864, 1998.
\bibitem{Nedelec}N\'ed\'elec, J. C., \textit{Acoustic and electromagnetic equations: integral representations for harmonic problems}, Springer Science \& Business Media, 2001.
\end{thebibliography}
\end{document}